\documentclass[a4paper, 12pt, reqno]{{amsart}}
\usepackage{amsmath, mathtools}
\usepackage{amscd}
\usepackage{amssymb}
\usepackage{amsthm}
\usepackage{ascmac, color}
\usepackage[dvipdfmx]{graphicx, xcolor, pdfpages}
\usepackage{mathrsfs}
\usepackage[all]{xy}
\usepackage[dvipdfmx]{hyperref}

\newtheorem{lemma}{{\sc Lemma}}[section]
\newtheorem{corollary}[lemma]{{\sc Corollary}}
\newtheorem{proposition}[lemma]{{\sc Proposition}}
\newtheorem{theorem}[lemma]{{\sc Theorem}}
\theoremstyle{definition}
\newtheorem{remark}[lemma]{{\sc Remark}}

\numberwithin{equation}{section}

\def\Gb{{\mathfrak{b}}}
\def\Gc{{\mathfrak{c}}}
\def\Bd{{\mathfrak{d}}}
\def\Gg{{\mathfrak{g}}}
\def\Gh{{\mathfrak{h}}}

\def\Gn{{\mathfrak{n}}}

\def\BA{{\mathbf{A}}}
\def\BB{{\mathbf{B}}}
\def\BC{{\mathbf{C}}}
\def\BD{{\mathbf{D}}}
\def\BE{{\mathbf{E}}}
\def\BF{{\mathbf{F}}}

\def\BQ{{\mathbf{Q}}}

\def\BZ{{\mathbf{Z}}}

\def\Bd{{\mathbf{d}}}

\def\Bm{{\mathbf{m}}}
\def\Bs{{\mathbf{s}}}
\def\Bt{{\mathbf{t}}}

\def\CB{{\mathcal B}}

\def\DD{{\mathcal D}}

\def\CO{{\mathcal O}}

\def\CM{{\mathcal M}}

\def\CR{{\mathcal R}}
\def\CS{{\mathcal S}}
\def\CT{{\mathcal T}}

\def\ad{{\mathop{\rm ad}\nolimits}}
\def\alg{{\mathop{\rm alg}\nolimits}}

\def\deru{\partial}

\def\End{\mathop{\rm{End}}\nolimits}

\def\gr{\mathop{\rm gr}\nolimits}
\def\Hom{\mathop{\rm Hom}\nolimits}

\def\Ht{\mathop{\rm ht}\nolimits}
\def\id{\mathop{\rm id}\nolimits}

\def\Ker{\mathop{\rm Ker\hskip.5pt}\nolimits}
\def\Mod{\mathop{\rm Mod}\nolimits}

\def\op{{\mathop{\rm op}\nolimits}}

\def\Proj{\mathop{\rm Proj}\nolimits}

\def\res{{\mathop{\rm res}\nolimits}}

\def\Spec{{\rm{Spec}}}

\def\Tor{{\rm{Tor}}}

\def\Dqab{D_q^{ab}}
\def\BDqab{{\BD_q^{ab}}}
\def\zDq{{}^0\!D_q}

\begin{document}
\title[differential operators on a quantized flag manifold]
{
The ring of differential operators on a quantized flag manifold}
\author{Toshiyuki TANISAKI}
\begin{abstract}
We establish some properties of the ring of differential operators on the  quantized flag manifold.
Especially, we give an explicit description of its localization on an affine open subset
in terms of the quantum Weyl algebra ($q$-analogue of boson).
\end{abstract}
\maketitle

\section{Introduction}
\subsection{}
Let $G$ be a connected, simply-connected simple algebraic group over the complex number field $\BC$, and let $\Gg$ be its Lie algebra.
The quantized flag manifold $\CB_q$ is a $q$-analogue of the flag manifold $\CB=B^-\backslash G$, where 
$B^-$ is a Borel subgroup of $G$.
Note that the ordinary flag manifold $\CB$ is a projective algebraic variety with homogeneous coordinate algebra
\[
A=\{f\in O(G)\mid f(ng)=f(g)\;(g\in G, n\in N^-)\},
\]
where $O(G)$ is the affine coordinate algebra of $G$, and $N^-$ denotes the unipotent radical of $B^-$.
A natural $q$-analogue $A_q$ of $A$ is defined 
as a subalgebra of the quantized coordinate algebra $O_q(G)$, and  the quantized flag manifold is defined by
 \[
 \CB_q=\Proj(A_q).
 \]
Since $A_q$ is a non-commutative graded ring, 
$\CB_q$ is a non-commutative algebraic variety 
(see \cite{LRD}, \cite{LR}).
In \cite{TBB}, \cite{T1},  \cite{T2}, \cite{T3}, \cite{TK}, \cite{T4}, \cite{T5}, we developed the theory of $D$-modules on 
$\CB_q$ using the framework presented  in \cite{LR}.
Namely, 
we defined a certain abelian category $\Mod(\DD_q)$  using  a certain graded subring $D_q$ of the ring of linear endomorphism of $A_q$, and investigated its connection with the representation theory of 
the quantized enveloping algebra of $\Gg$.
The category $\Mod(\DD_q)$ is a $q$-analogue of the category $\Mod(\DD)$ consisting of quasi-coherent $\DD$-modules, where $\DD$ is the sheaf of rings of universally twisted differential operators on $\CB$.
The present paper is devoted to establishing some fundamental properties of the ring $D_q$ itself.

Let $\Lambda$ be the weight lattice and let $\Lambda^+$ be the set of dominant weights.
Let $U_q(\Gg)$ be the quantized enveloping algebra of $\Gg$.
It is an associative algebra over the field $\BF=\BQ(q^{1/N})$ generated by elements
$e_i, f_i$ ($i\in I$), $k_\lambda$ 
($\lambda\in\Lambda$)  satisfying standard relations.
Here, $N$ is a positive integer satisfying a certain condition.
The $\BF$-algebra $A_q$ has a $\Lambda$-grading 
\[
A_q=\bigoplus_{\nu\in\Lambda^+}A_q(\nu).
\]
We have a natural left $U_q(\Gg)$-module structure of $A_q$ respecting the grading.
For $u\in U_q(\Gg)$, $\varphi\in A_q$, 
$\lambda\in\Lambda$ we define linear endomorphisms 
$\deru_u$, $\ell_\varphi$, $\sigma_\lambda$ of $A_q$ as the left action of $u$, the left multiplication by $\varphi$ and the grading operator given by
$\sigma_\lambda(\psi)=q^{(\lambda,\nu)}\psi$ for $\psi\in A_q(\nu)$ respectively.
Then we define $D_q$ to be the subalgebra generated by $\deru_u$ ($u\in U_q(\Gg)$), $\ell_\varphi$ ($\varphi\in A_q$), $\sigma_\lambda$ ($\lambda\in\Lambda$).

We have a set of obvious exchange relations among the above generators of $D_q$ (see \eqref{subeq:rel} below).
On the other hand we have another set of non-obvious relations given in the following manner.
For $\varphi\in A_q$ let $r_\varphi$ be the linear endomorphism of $A_q$ given by the right multiplication  by $\varphi$.
Since $A_q$ is non-commutative, $r_\varphi$ does not coincide with $\ell_\varphi$.
Note that $A_q$ is defined as a subalgebra of $O_q(G)$.
Since $O_q(G)$ is a Hopf algebra dual to $U_q(\Gg)$, 
the non-commutativity of $O_q(G)$ comes from the non-cocomutativity of $U_q(\Gg)$, which is controlled by the universal $R$-matrix $\CR$.
This consideration leads us to express 
$r_\varphi$ as an element $\Xi_\CR(\varphi)\in D_q$ involving $\CR$ and the elements
$\deru_u$ ($u\in U_q(\Gg)$), $\ell_\psi$ ($\psi\in A_q$), $\sigma_\lambda$ ($\lambda\in\Lambda$).
Note that we have two universal $R$-matrices $\CR$ and $\CR'$, one in a completion of $U_q(\Gb^+)\otimes U_q(\Gb^-)$ and the other in a completion of $U_q(\Gb^-)\otimes U_q(\Gb^+)$.
Here, $\Gb^+$ and $\Gb^-$ are opposite Borel subalgebras of $\Gg$.
Hence, we obtain a non-obvious relation 
$\Xi_\CR(\varphi)=\Xi_{\CR'}(\varphi)$ in $D_q$ 
for each $\varphi\in A_q$
(see Proposition \ref{prop:Rel1} below).
Those non-obvious relations play crucial roles in this paper.

Let $U^e_q(\Gg)$ be the subalgebra of $U_q(\Gg)$ generated by $e_i, f_ik_{\alpha_i}$ ($i\in I$), $k_{2\lambda}$ ($\lambda\in\Lambda$).
Then we have
\[
U_q(\Gg)=\bigoplus_{\overline{\lambda}\in\Lambda/2\Lambda}k_\lambda U_q^e(\Gg)
=
\bigoplus_{\overline{\lambda}\in\Lambda/2\Lambda}U_q^e(\Gg)k_\lambda.
\]
For $a=1, 2$ we set
\[
U_q(\Gg)^a=
\begin{cases}
U_q(\Gg)\quad&(a=1)
\\
U_q^e(\Gg)\quad&(a=2).
\end{cases}
\]
For $a, b=1, 2$ we define a subalgebra $D_q^{ab}$ of $D_q$ by
\[
D_q^{ab}=\langle \deru_u, \ell_\varphi, \sigma_{\lambda}
\mid
u\in U_q(\Gg)^a, \varphi\in A_q, \lambda\in b\Lambda
\rangle.
\]
In particular, we have $D_q=D_q^{11}$.
One of the main results of this paper is 
\begin{subequations}
\begin{align}
D_q^{a1}\cong&\BF[\sigma_\Lambda]\otimes_{\BF[\sigma_{2\Lambda}]}D_q^{a2}
\cong
D_q^{a2}\otimes_{\BF[\sigma_{2\Lambda}]}
\BF[\sigma_\Lambda]
&(a=1, 2),
\\
D_q^{1b}
\cong&
U_q(\Gg)\otimes_{U_q^e(\Gg)}D_q^{2b}
\cong
D_q^{2b}\otimes_{U_q^e(\Gg)}
U_q(\Gg)
&(b=1, 2),
\end{align}
\end{subequations}
where $\BF[\sigma_{b\Lambda}]$ 
($b=1, 2$) is the subalgebra 
$\sum_{\lambda\in\Lambda}\BF\sigma_{b\lambda}$ 
of $D_q$, which is naturally isomorphic to the group algebra of $b\Lambda$.
We will prove this fact through some consideration on the adjoint action of $U_q(\Gg)$ on $D_q$.

Let $W$ be the Weyl group.
For $w\in W$ set $U_w=B^-\backslash B^-B^+w\subset \CB$, where $B^+$ is the Borel subgroup opposite to $B^-$.
The affine open covering 
$
\CB=\bigcup_{w\in W}U_w
$
of $\CB$ admits a natural $q$-analogue 
\[
\CB_q=\bigcup_{w\in W}U_{w,q}.
\]
For each $w\in W$ 
we can define a ring
${\DD}_q|_{U_{w,q}}$ 
as the degree zero part of the  localization of 
${D}_q$ with respect to a certain multiplicative subset $\CS_w$ satisfying the left and right Ore conditions.
We give an explicit description of the ring 
${\DD}_q|_{U_{w,q}}$ in this paper.
Using the braid group action we can show 
$
{\DD}_q|_{U_{w,q}}
\cong
{\DD}_q|_{U_{1,q}}
$, 
and hence it is sufficient to give an explicit description of ${\DD}_q|_{U_{1,q}}$.
This ring contains a subring $\BB_q$ which already appeared in \cite{Jo0}.
This  subring  is isomorphic to the ring called the 
$q$-analogue of  boson in \cite{Kas}, and 
the quantum Weyl algebra in \cite{JoB}.
For $\lambda\in\Lambda$ set $t_\lambda=\deru_{k_{-\lambda}}\sigma_\lambda\in(\CS_1^{-1}D_q)(0)=\DD_q|_{U_{1.q}}$.
Then the subalgebra 
$\BF[t_\Lambda]$ 
of $\DD_q|_{U_{1.q}}$ generated by 
$t_\lambda$ ($\lambda\in\Lambda$)  is 
naturally isomorphic to the group algebra of 
$\Lambda$.
Using the non-obvious relations in $D_q$ we can show 
$t_{\lambda}\in\BB_q$ for $\lambda\in 2\Lambda^+$.
We have the following explicit description of $\DD_q|_{U_{1.q}}$:
\begin{equation}
\DD_q|_{U_{1.q}}
\cong
(\BF[t_\Lambda]\otimes_{\BF[t_{2\Lambda^+}]}
\BB_q)\otimes_\BF\BF[\sigma_\Lambda]
\cong
(\BB_q\otimes_{\BF[t_{2\Lambda^+}]}
\BF[t_\Lambda])\otimes_\BF\BF[\sigma_\Lambda]
\end{equation}

Let $\tilde{D}_q$ be the algebra defined by the symbols 
$\deru_u$ ($u\in U_q(\Gg)$), $\ell_\varphi$ ($\varphi\in A_q$), $\sigma_\lambda$ ($\lambda\in\Lambda$) satisfying the obvious and the non-obvious relations as above.
We will also show that 
the natural surjective algebra homomorphism
$\tilde{p}:\tilde{D}_q\to D_q$ is almost isomorphic.
More precisely, we show that $\tilde{p}$ induces the 
isomorphism
\begin{equation}
\tilde{\DD}_q|_{U_{w,q}}
\cong
{\DD}_q|_{U_{w,q}}
\end{equation}
of algebras for any $w\in W$.
From this we obtain the equivalence 
\begin{equation}
\Mod(\DD_q)\cong\Mod(\tilde{\DD}_q)
\end{equation}
of abelian categories 
 defined using 
$D_q$ and $\tilde{D}_q$.
\subsection{}
The contents of this paper is as follows.
In Section 2 we recall  well-known standard results on the quantized enveloping algebras and the quantized coordinate algebras.
In Section 3 we introduce the ring of differential operators on the quantized flag manifold.
In Section 4 we establish some relations among certain subrings.
In Section 5 we give a local description of this ring.
Some applications to the theory of $D$-modules on the quantized flag manifold is given in Section 6.

\subsection{}
For a Hopf algebra $H$ we use Sweedler's notation 
\[
\Delta(h)=\sum_{(h)}h_{(0)}\otimes h_{(1)}
\qquad(h\in H)
\]
for the comultiplication $\Delta:H\to H\otimes H$.
\section{Quantum groups}
\subsection{}
Let $\Gg$ be a finite-dimensional simple Lie algebra over $\BC$, and let $\Gh$ be its Cartan subalgebra.
We denote by 
\[
\Delta\subset Q\subset \Lambda\subset\Gh^*
\]
the set of roots, the root lattice, the weight lattice respectively.
For $\alpha\in\Delta$ the corresponding coroot is denoted by $\alpha^\vee\in\Gh$.
We denote by $W\subset GL(\Gh^*)$ the Weyl group.
Let 
\[
(\;,\:):\Gh^*\times\Gh^*\to\BC
\]
be the $W$-invariant symmetric bilinear form satisfying $(\alpha,\alpha)=2$ for any short root $\alpha$.
We fix a set of simple roots $\{\alpha_i\}_{i\in I}\subset\Delta$, and denote by $\Delta^+$ the corresponding set of positive roots.
We denote by $s_i\in W$ 
the reflection associated to $i\in I$.
We set
\[
Q^+=\sum_{\alpha\in\Delta^+}\BZ_{\geqq0}\alpha\subset Q,
\qquad
\Lambda^+=\{\lambda\in\Lambda\mid
\langle\lambda,\alpha^\vee\rangle\geqq0\;(\forall\alpha\in\Delta^+)\}\subset\Lambda.
\]
For $i, j\in I$ we set $a_{ij}=\langle\alpha_j,\alpha_i^\vee\rangle$.
We define subalgebras $\Gb^+$, $\Gb^-$, $\Gn^+$, $\Gn^-$ of $\Gg$ by
\[
\Gn^\pm=\bigoplus_{\alpha\in\Delta^+}\Gg_{\pm\alpha},
\qquad
\Gb^\pm=\Gh\oplus\Gn^\pm,
\]
where $\Gg_\alpha$ for $\alpha\in\Delta$ denotes the corresponding root subspace.
\subsection{}
Take a positive integer $N$ satisfying
\begin{equation}
N(\Lambda,\Lambda)\subset\BZ,
\end{equation}
and let
\begin{equation}
\BF=\BQ(q^{1/N})
\end{equation}
be the rational function field in variable $q^{1/N}$ over $\BQ$.
We define the quantized enveloping algebra $U_q(\Gg)$ to be the associative $\BF$-algebra with $1$ generated by the elements
\[
k_\lambda\quad(\lambda\in\Lambda),
\qquad
e_i, f_i\quad(i\in I)
\]
satisfying the relations
\begin{subequations}
\begin{align}
&k_0=1, 
\quad k_\lambda k_\mu=k_{\lambda+\mu}
&(\lambda, \mu\in\Lambda),
\\
&
k_\lambda e_i=q_i^{\langle\lambda,\alpha_i^\vee\rangle}e_i k_\lambda,
\quad
k_\lambda f_i=q_i^{-\langle\lambda,\alpha_i^\vee\rangle}f_i k_\lambda
&(\lambda\in \Lambda, \;i\in I),
\\
&e_if_j-f_je_i=
\delta_{ij}
\frac{k_i-k_i^{-1}}{q_i-q_i^{-1}}
&(i, j\in I), 
\\
&
\sum_{r=0}^{1-a_{ij}}
(-1)^re_i^{(1-a_{ij}-r)}e_je_i^{(r)}
=
\sum_{r=0}^{1-a_{ij}}
(-1)^rf_i^{(1-a_{ij}-r)}f_jf_i^{(r)}=0
&(i, j\in I, \; i\ne j),
\end{align}
\end{subequations}
where $q_i=q^{(\alpha_i,\alpha_i)/2}$, $k_i=k_{\alpha_i}$ for $i\in I$, and 
$e_i^{(n)}=e_i^n/[n]_{q_i}!$, 
$f_i^{(n)}=f_i^n/[n]_{q_i}!$ 
for $i\in I$, $n\in\BZ_{\geqq0}$.
Here, for $n\in\BZ_{\geqq0}$ the Laurent polynomials 
$[n]_t, [n]_t!\in\BZ[t,t^{-1}]$ are defined by
\[
[n]_t=\frac{t^n-t^{-n}}{t-t^{-1}},
\qquad
[n]_t!=[1]_t[2]_t\dots[n]_t.
\]
Then $U_q(\Gg)$ is endowed with a Hopf algebra structure by
the counit $\varepsilon$, the comulitiplication $\Delta$, the antipode $S$ given by
\begin{subequations}
\begin{align}
&\varepsilon(k_\lambda)=1\quad(\lambda\in\Lambda),
\qquad
\varepsilon(e_i)=\varepsilon(f_i)=0\quad
(i\in I),
\\
&\Delta(k_\lambda)=k_\lambda\otimes k_\lambda
\quad(\lambda\in\Lambda),
\\
&
\nonumber
\Delta(e_i)=e_i\otimes1+k_i\otimes e_i,
\quad
\Delta(f_i)=f_i\otimes k_i^{-1}\otimes 1\otimes f_i\quad
(i\in I),
\\
&
S(k_\lambda)=k_{-\lambda}
\quad(\lambda\in\Lambda),
\qquad
S(e_i)=-k_i^{-1}e_i,\quad
S(f_i)=-f_ik_i\quad
(i\in I).
\end{align}
\end{subequations}

We define the adjoint action of $U_q(\Gg)$ on $U_q(\Gg)$ by
\[
\ad(u)(u')=\sum_{(u)}u_{(0)}u'(Su_{(1)})
\qquad(u, u'\in U_q(\Gg)).
\]

We define subalgebras $U_q(\Gh)$, $U_q(\Gb^\pm)$, $U_q(\Gn^\pm)$ of $U_q(\Gg)$ by
\[
U_q(\Gh)=\langle k_\lambda\mid\lambda\in\Lambda\rangle, 
\qquad
\]
\[
U_q(\Gn^+)=\langle e_i\mid i\in I\rangle,
\qquad
U_q(\Gn^-)=\langle f_i\mid i\in I\rangle,
\qquad
U_q(\Gb^\pm)=\langle U_q(\Gh), U_q(\Gn^\pm)\rangle.
\]
The subalgebas $U_q(\Gh)$, $U_q(\Gb^\pm)$ are Hopf subalgebras.
The multiplication of $U_q(\Gg)$ gives the isomorphisms
\[
U_q(\Gn^+)\otimes U_q(\Gh)\otimes U_q(\Gn^-)\xrightarrow{\sim}U_q(\Gg)
\xleftarrow{\sim}
U_q(\Gn^-)\otimes U_q(\Gh)\otimes U_q(\Gn^+)
\]
of $\BF$-modules.
We have
\[
U_q(\Gh)=\bigoplus_{\lambda\in\Lambda}\BF k_\lambda.
\]
For $\lambda\in\Lambda$ we define an algebra homomorphism 
\begin{equation}
\chi_\lambda:U_q(\Gh)\to\BF
\end{equation}
by $\chi_\lambda(k_\mu)=q^{(\lambda,\mu)}$ for $\mu\in\Lambda$.
Setting 
\[
U_q(\Gn^\pm)_\gamma=\{u\in U_q(\Gn^\pm)\mid
\ad(h)(u)=\chi_\gamma(h)u\;(h\in U_q(\Gh))\}
\qquad(\gamma\in Q)
\]
we have
\[
U_q(\Gn^\pm)=
\bigoplus_{\gamma\in Q^+}
U_q(\Gn^\pm)_{\pm\gamma}.
\]

Set
\begin{gather*}
U^e_q(\Gh)=\sum_{\lambda\in\Lambda}\BF k_{2\lambda}\subset U_q(\Gh),
\qquad
\tilde{U}_q(\Gn^-)=S(U_q(\Gn^-))=\langle f_ik_i\mid i\in I\rangle,
\\
U^e_q(\Gg)=\langle U^e_q(\Gh),\tilde{U}_q(\Gn^-), U_q(\Gn^+)\rangle.
\end{gather*}
The multiplication of $U^e_q(\Gg)$ gives the isomorphisms
\[
U_q(\Gn^+)\otimes U^e_q(\Gh)\otimes \tilde{U}_q(\Gn^-)\xrightarrow{\sim}U^e_q(\Gg)
\xleftarrow{\sim}
\tilde{U}_q(\Gn^-)\otimes U^e_q(\Gh)\otimes U_q(\Gn^+)
\]
of $\BF$-modules.
For $\lambda\in\Lambda$ we have
$k_\lambda U_q^e(\Gg)
=U_q^e(\Gg)k_\lambda$.
Moreover, this subspace of $U_q(\Gg)$ depends only on $\overline{\lambda}\in\Lambda/2\Lambda$, and 
we have
\begin{equation}
\label{eq:UUe}
U_q(\Gg)
=
\bigoplus_{\overline{\lambda}\in\Lambda/2\Lambda}
k_\lambda U_q^e(\Gg)
=
\bigoplus_{\overline{\lambda}\in\Lambda/2\Lambda}
U_q^e(\Gg)k_\lambda.
\end{equation}
We have also
\begin{equation}
\label{eq:adUe}
\ad(U_q(\Gg))(U_q^e(\Gg)k_\lambda)\subset
U_q^e(\Gg)k_\lambda
\end{equation}
for any $\lambda\in\Lambda$.

Set 
\[
U_q^f(\Gg)=
\{u\in U_q(\Gg)\mid
\dim \ad(U_q(\Gg))(u)<\infty\}.
\]
It is a subalgebra of $U_q(\Gg)$.
By \cite{JoB} we have $k_{-2\lambda}\in U_q^f(\Gg)$ for $\lambda\in\Lambda^+$, and 
the multiplicative set  
$\CT=\{k_{-2\lambda}\mid\lambda\in\Lambda^+\}$ 
satisfies the left and right Ore conditions in $U_q^f(\Gg)$.
Moreover, we have
\begin{equation}
\label{eq:Ufe}
\CT^{-1}U_q^f(\Gg)=U_q^e(\Gg).
\end{equation}

\subsection{}
We denote by 
\begin{equation}
\label{eq:Dr}
\tau:U_q(\Gb^+)\times U_q(\Gb^-)\to\BF
\end{equation}
the bilinear form uniquely characterized by
\begin{subequations}
\begin{align}
&\tau(x,y_1y_2)=(\tau\otimes\tau)(\Delta(x),y_1\otimes y_2)
&(x\in U_q(\Gb^+),\,y_1,y_2\in U_q(\Gb^-)),\\
&\tau(x_1x_2,y)=(\tau\otimes\tau)(x_2\otimes x_1,\Delta(y))
&(x_1, x_2\in U_q(\Gb^+),\,y\in U_q(\Gb^-)),\\
&\tau(k_\lambda,k_\mu)=q^{-(\lambda,\mu)}
&(\lambda,\mu\in\Lambda),\\
&\tau(k_\lambda, f_i)=\tau(e_i,k_\lambda)=0
&(\lambda\in\Lambda,\,i\in I),\\
&\tau(e_i,f_j)=\delta_{ij}/(q_i^{-1}-q_i)
&(i,j\in I).
\end{align}
\end{subequations}
We have
\begin{align}
&\tau(xk_\lambda, yk_\mu)=q^{-(\lambda,\mu)}\tau(x,y)
\qquad&(x\in U_q(\Gn^+), y\in U_q(\Gn^-), \lambda, \mu\in\Lambda),
\\
&\tau(U_q(\Gn^+)_\gamma, U_q(\Gn^-)_{-\delta})=0
&(\gamma, \delta\in Q^+, \gamma\ne\delta).
\end{align}
Moreover, for any $\gamma\in Q^+$ the restriction 
$\tau|_{U_q(\Gn^+)_\gamma\times U_q(\Gn^-)_{-\gamma}}$ is non-degenerate
(see \cite{T0}).
\subsection{}
For $i\in I$ let $T_i$ be Lusztig's automorphism
of the $\BF$-algebra $U_q(\Gg)$ given by 
\begin{align}
T_i(k_\mu)=&k_{s_i(\mu)}
\quad&(\mu\in\Lambda),
\\
T_i(e_j)=&
\begin{cases}
\sum_{k=0}^{-a_{ij}}(-q_i)^{-k}
e_i^{(-a_{ij}-k)}e_je_i^{(k)}
&(j\ne i)
\\
-f_ik_i
&(j=i)
\end{cases}
&(j\in I),
\\
T_i(f_j)=&
\begin{cases}
\sum_{k=0}^{-a_{ij}}(-q_i)^{k}
f_i^{(k)}f_jf_i^{(-a_{ij}-k)}
&(j\ne i)
\\
-k_i^{-1}e_i
&(j=i)
\end{cases}
&(j\in I).
\end{align}
We fix a reduced expression $w_0=s_{i_1}\dots s_{i_L}$ of the longest element $w_0\in W$ and set
\begin{equation}
\label{eq:root}
\beta_t=s_{i_1}\dots s_{i_{t-1}}(\alpha_{i_t}),
\qquad
e_{\beta_t}=T_{i_1}\dots T_{i_{t-1}}(e_{i_t})
\end{equation}
for $t=1,\dots, L$.
Then we have $\Delta^+=\{\beta_1,\dots, \beta_L\}$, and 
the monomials
$
e_{\beta_1}^{k_1}\dots e_{\beta_L}^{k_L}
$
for $(k_1,\dots, k_L)\in\BZ_{\geqq0}^L$ 
form a basis of $U_q(\Gn^+)$
(see \cite{Lb}).

\subsection{}
For a left (resp.\ right) $U_q(\Gh)$-module $M$ and $\lambda\in\Lambda$ we denote by $M_\lambda$ 
the subspace consisting of $m\in M$ satisfying 
$hm=\chi_\lambda(h)m$
(resp. $mh=\chi_\lambda(h)m$)
for any $h\in U_q(\Gh)$.
We say that a left or right $U_q(\Gh)$-module $M$ is 
diagonalizable 
if it is a direct sum of weight spaces $M_\lambda$ for $\lambda\in\Lambda$.

Let $\Gc$ be one of $\Gg, \Gb^+, \Gh$, and 
set $C=G, B^+, H$ accordingly.
We say that a left (resp.\ right) $U_q(\Gc)$-module $M$ is integrable if it is diagonalizable as a left (resp.\ right) $U_q(\Gh)$-module and we have
$\dim U_q(\Gc)m<\infty$
(resp.\ $\dim mU_q(\Gc)<\infty$) 
for any $m\in M$.
Note that the dual space $U_q(\Gc)^*=\Hom_\BF(U_q(\Gc),\BF)$ is endowed with a $U_q(\Gc)$-bimodule structure by
\[
\langle u\cdot\varphi\cdot u',u''\rangle
=
\langle\varphi,u'u''u\rangle
\qquad(\varphi\in U_q(\Gc)^*,  u, u', u''\in U_q(\Gc)).
\]
Let $O_q(C)$ be the maximal integrable left $U_q(\Gc)$-submodule of $U_q(\Gc)^*$.
Then it is also the maximal integrable right $U_q(\Gc)$-submodule of $U_q(\Gc)^*$.
We have a Hopf algebra structure of $O_q(C)$ whose multiplication, unit, comultiplication, counit, antipode are given by the transposes of 
comultiplication, counit, multiplication, unit, antipode respectively 
(see for example \cite{TBB}).
We have 
\[
O_q(H)=\bigoplus_{\lambda\in\Lambda}\BF\chi_\lambda.
\]

The embedding
$U_q(\Gh)\hookrightarrow U_q(\Gb^+)\hookrightarrow
U_q(\Gg)$ induces 
 surjective Hopf algebra homomorphisms
\[
O_q(G)\xrightarrow{\res}O_q(B^+)\xrightarrow{\res}O_q(H).
\]
On the other hand 
the surjective Hopf algebra homomorphism
\[
\pi:U_q(\Gb^+)\to U_q(\Gh)
\]
given by 
\[
\pi(h)=h\quad(h\in U_q(\Gh)),
\qquad
\pi(x)=\varepsilon(x)\quad(x\in U_q(\Gn^+))
\]
induces an injective Hopf algebra homomorphism
\begin{equation}
\label{eq:pi*}
\pi^*:O_q(H)\hookrightarrow O_q(B^+).
\end{equation}

\section{The algebra of differential operators}
\subsection{}
We define a subalgebra $A_q$ of $O_q(G)$ by
\[
A_q=\{\varphi\in O_q(G)\mid
\varphi\cdot y=\varepsilon(y)\varphi\;(\forall y\in U_q(\Gn^-))\}.
\]
It is a $(U_q(\Gg),U_q(\Gh))$-submodule of the $U_q(\Gg)$-bimodule $O_q(G)$.
Setting
\[
A_q(\nu)=\{\varphi\in A_q\mid
\varphi\cdot h=\chi_\nu(h)\varphi\;(\forall h\in U_q(\Gh))\}
\]
for $\lambda\in\Lambda^+$
we have
\begin{equation}
\label{eq:grA}
A_q=\bigoplus_{\nu\in\Lambda^+}A_q(\nu).
\end{equation}
Moreover, $A_q$ turns out to be a $\Lambda$-graded algebra by \eqref{eq:grA}.
It follows from the Peter-Weyl type theorem for $O_q(G)$ that $A_q(\nu)$ is an irreducible left $U_q(\Gg)$-module with highest weight $\nu$.
By \cite{Jo0} $A_q$ is an integral  domain 
in the sense that $\varphi\psi=0$ for $\varphi, \psi\in A_q$ implies $\varphi=0$ or $\psi=0$.

For $w\in W$ set 
\[
\CS_w=\bigcup_{\nu\in\Lambda^+}(A_q(\nu)_{w^{-1}\nu}\setminus\{0\}).
\]
Here, $A_q(\nu)_\mu$ for $\mu\in\Lambda$ denotes the weight space with respect to the left $U_q(\Gh)$-module structure.
By Joseph \cite{Jo0} the multiplicative subset $\CS_w$ of $A_q$ satisfies the left and right Ore conditions.
Since $A_q$ is an integral  domain, the canonical homomorphism $A_q\to \CS_w^{-1}A_q$ is injective.

\subsection{}
For $\varphi\in A_q$, $u\in U_q(\Gg)$, $\lambda\in \Lambda$ we define operators
$
\ell_{\varphi},
\deru_u,
\sigma_\lambda\in\End_\BF(A_q)
$
by
\[
\ell_{\varphi}(\psi)=\varphi\psi,
\qquad
\deru_u(\psi)=u\cdot\psi,
\qquad
\sigma_\lambda(\psi)=\psi\cdot k_\lambda
\]
for $\psi\in A_q$.
We have the following relations:
\begin{subequations}
\label{subeq:rel}
\begin{align}
\label{eq:rel0}
&\ell_1=\deru_1=\sigma_0=\id,
\\
\label{eq:rel1}
&\ell_\varphi\ell_\psi=\ell_{\varphi\psi}
&(\varphi, \psi\in {A_q}),
\\
\label{eq:rel2}
&
\deru_u\deru_{u'}=\deru_{uu'}
&(u, u'\in U_q(\Gg)),
\\
\label{eq:rel3}
&
\sigma_\lambda\sigma_{\mu}=\sigma_{\lambda+\mu}
&(\lambda, \mu\in\Lambda),
\\
\label{eq:rel4}
&\deru_u\ell_\varphi=\sum_{(u)}
\ell_{u_{(0)}\cdot\varphi}\deru_{u_{(1)}}
&(u\in U_q(\Gg),\; \varphi\in {A_q}),
\\
\label{eq:4a}
&\ell_\varphi\deru_u=
\sum_{(u)}
\deru_{u_{(1)}}\ell_{(S^{-1}u_{(0)})\cdot\varphi}
&
(u\in U_q(\Gg),\; \varphi\in {A_q}),
\\
\label{eq:rel5}
&\sigma_\lambda\deru_u=
\deru_u\sigma_\lambda
&(\lambda\in\Lambda,\;  u\in U_q(\Gg)),
\\
\label{eq:rel6}
&\sigma_\lambda\ell_\varphi=
q^{(\lambda,\nu)}\ell_\varphi\sigma_{\lambda}
&(\lambda\in\Lambda,\; \nu\in\Lambda^+,\;\varphi\in {A_q}(\nu)).
\end{align}
\end{subequations}

We define an $\BF$-subalgebra
$D_q$ of $\End_\BF(A_q)$ by
\begin{align}
D_{q}&=
\langle \ell_\varphi, \deru_u, \sigma_{\lambda}\mid
\varphi\in A_q, u\in U_q(\Gg), \lambda\in \Lambda
\rangle.
\end{align}
For $a, b=1, 2$ we define an $\BF$-subalgebra
$D_q^{ab}$ of $D_q$ 
 by
 \begin{align}
D_q^{ab}=
\langle \ell_\varphi, \deru_u, \sigma_{\lambda}\mid
\varphi\in A_q, u\in U_q(\Gg)^a, \lambda\in b\Lambda
\rangle,
\end{align}
where
\[
U_q(\Gg)^a=
\begin{cases}
U_q(\Gg)\quad&(a=1)
\\
U_q^e(\Gg)\quad&(a=2).
\end{cases}
\]

We take 
a basis $\{x_r\}$ of $U_q(\Gn^+)$,
a basis $\{y_r\}$ of $U_q(\Gn^-)$,
$\beta_r\in Q^+$ such that
\begin{equation}
\label{eq:canon}
x_r\in U_q(\Gn^+)_{\beta_r},
\qquad
y_r\in U_q(\Gn^-)_{-\beta_r},
\qquad
\tau(x_r, y_s)=\delta_{rs}.
\end{equation}
By \cite[Lemma 4.1]{T1} we have the following.
\begin{proposition}
\label{prop:Rel1}
Let $\nu\in\Lambda^+$ and $\mu\in\Lambda$.
Then for any $\varphi\in A_q(\nu)_\mu$ we have
\begin{equation}
\label{eq:Rel1}
\left(
\sum_r\ell_{y_r\cdot\varphi}\deru_{x_r}
\right)
\deru_{k_{-2\mu}}\sigma_{2\nu}
=
\sum_r\ell_{(Sx_r)\cdot\varphi}\deru_{y_rk_{\beta_r}}
\end{equation}
in $D_q^{22}$.
\end{proposition}

We have a grading
\begin{equation}
\label{eq:Dgr}
D_q^{ab}=\bigoplus_{\nu\in\Lambda^+}D_q^{ab}(\nu)
\qquad
(a, b=1,2)
\end{equation}
of $D_q^{ab}$ given by
\[
D_q^{ab}(\nu)=\{P\in D_q^{ab}\mid P(A_q(\nu'))\subset A_q(\nu'+\nu)
\;(\forall\nu'\in\Lambda^+)\}.
\]

\subsection{}
We sometimes regard $A_q$ with a subalgebra of $D_q^{ab}$ ($a, b=1, 2$) by the embedding $A_q\ni\varphi\mapsto \ell_\varphi\in D_q^{ab}$.

For $w\in W$ the multiplicative subset $\CS_w$ of $D_q^{ab}$ 
 ($a, b=1, 2$) satisfies the left and right Ore conditions.
 Moreover,   the canonical algebra homomorphism $D_q^{ab}\to \CS_w^{-1}D_q^{ab}$ is injective (see \cite[Section 1.2]{LRD}, \cite[Proposition 5.1]{TBB}).
 
For $\lambda\in\Lambda^+$ we define $c_\lambda\in A_q(\lambda)_\lambda$ by 
$
\langle c_\lambda,1\rangle =1$.
Then we have
$c_\lambda c_\mu=c_{\lambda+\mu}$ for $\lambda, \mu\in\Lambda^+$ and $\CS_1=\bigsqcup_{\lambda\in\Lambda^+}\BF^\times c_\lambda$.
It is easily seen that  
the grading \eqref{eq:Dgr} induces the grading
 \begin{equation}
 \label{eq:SDgr}
 \CS_1^{-1}D_q^{ab}
 =
 \bigoplus_{\nu\in\Lambda}
 (\CS_1^{-1}D_q^{ab})(\nu)
 \end{equation}
 of $\CS_1^{-1}D_q^{ab}$.
We regard $D_q^{ab}$ as a subalgebra of $\CS_1^{-1}D_q^{ab}$.
Then $\CS_1^{-1}D_q^{ab}$ is generated by 
$\ell_\varphi$ ($\varphi\in \CS_1^{-1}A_q$), 
$\deru_u$ ($u\in U_q(\Gg)^a$), 
$\sigma_\lambda$ ($\lambda\in b\Lambda$),
where for $\varphi=c^{-1}\psi\in \CS_1^{-1}A_q$ with 
$c\in\CS_1$, $\psi\in A_q$ we set 
$\ell_\varphi=\ell_c^{-1}\ell_\psi$.

 We have
 \[
 \CS_1^{-1}\Dqab\cong \CS_1^{-1}A_q\otimes_{A_q}\Dqab\cong \Dqab\otimes_{A_q}\CS_1^{-1}A_q
 \cong
\CS_1^{-1}A_q\otimes_{A_q}\Dqab\otimes_{A_q}\CS_1^{-1}A_q,
 \]
and hence $\CS_1^{-1}A_q$ is a $\Lambda$-graded $\CS_1^{-1}\Dqab$-module.
By \cite[Lemma 5.7]{TBB}
the canonical homomorphism
\begin{equation}
\CS_1^{-1}\Dqab\to\End_\BF(\CS_1^{-1}A_q)
\end{equation}
is injective.
We sometimes regard $\CS_1^{-1}\Dqab$ 
as a subalgebra of $\End_\BF(\CS_1^{-1}A_q)$.

The subalgebra $(\CS_1^{-1}\Dqab)(0)$ 
of $\End_\BF(\CS_1^{-1}A_q)$
is generated by
${\ell}_\varphi, {\deru}_u,{\sigma}_{\lambda}\in\End_\BF(\CS_1^{-1}A_q)$ for
$\varphi\in (\CS_1^{-1}A_q)(0)$, $u\in U_q(\Gg)^a$, $\lambda\in b\Lambda$.
We have
\begin{align}
\label{eq:action1}
{\ell}_\varphi(\psi)=&\varphi\psi
\qquad&(\varphi\in (\CS_1^{-1}A_q)(0), \psi\in \CS_1^{-1}A_q),
\\
\label{eq:action2}
{\sigma}_{\lambda}(\psi)=&q^{(\lambda,\nu)}\psi
\qquad&(\lambda\in b\Lambda, \psi\in (S^{-1}A)(\nu)).
\end{align}
An  explicit description of ${\deru}_u\in\End_\BF(\CS_1^{-1}A_q)$ for $u=k_\lambda, e_i, f_i$ is obtained using
\begin{align}
\label{eq:action3}
{\deru}_u(\psi\psi')=&
\sum_{(u)}
{\deru}_{u_{(0)}}(\psi)
{\deru}_{u_{(1)}}(\psi')
\qquad
&(\psi, \psi'\in \CS_1^{-1}A_q).
\end{align}
For example, for $c\in\CS_1$ we have
\[
1=\deru_{k_\lambda}(1)=
{\deru}_{k_\lambda}(cc^{-1})
={\deru}_{k_\lambda}(c)
{\deru}_{k_\lambda}(c^{-1})
\]
and hence 
${\deru}_{k_\lambda}(c^{-1})=
({\deru}_{k_\lambda}(c))^{-1}$.
Therefore, for $c\in\CS_1$ and $\psi\in A_q$ we have
\[
{\deru}_{k_\lambda}(c^{-1}\psi)=
{\deru}_{k_\lambda}(c^{-1})
{\deru}_{k_\lambda}(\psi)
=
({\deru}_{k_\lambda}(c))^{-1}
{\deru}_{k_\lambda}(\psi).
\]
\begin{proposition}
\label{prop:Rel2}
\begin{itemize}
\item[(i)]
For $\lambda\in\Lambda^+$ we have
\[
{\deru}_{k_{2\lambda}}
{\sigma}_{-2\lambda}
=
\sum_r{\ell}_{c_\lambda^{-1}(y_r\cdot c_\lambda)}{\deru}_{x_r}
\in
\langle{\ell}_\varphi, {\deru}_x
\mid
\varphi\in(\CS_1^{-1}A_q)(0), 
x\in U_q(\Gn^+)
\rangle.
\]
\item[(ii)]
For $y\in\tilde{U}_q(\Gn^-)$ we have
\begin{align*}
{\deru}_y\in&
\langle{\ell}_\varphi, {\deru}_x, 
{\deru}_{k_{2\lambda}},{\sigma}_{2\mu}
\mid
\varphi\in(\CS_1^{-1}A_q)(0), 
x\in U_q(\Gn^+), 
\lambda, \mu\in \Lambda
\rangle.
\end{align*}
\end{itemize}
\end{proposition}
\begin{proof}
Setting $\varphi=c_\lambda$ in \eqref{eq:Rel1} we obtain
\[
\left(
\sum_r\ell_{y_r\cdot c_\lambda}\deru_{x_r}
\right)
\deru_{k_{-2\lambda}}\sigma_{2\lambda}
=
\ell_{c_\lambda}.
\]
Hence (i) holds.
To show (ii) we may assume that $y=f_ik_i$ for some $i\in I$.
Then the desired result follows by 
setting $\varphi=f_i\cdot c_\nu$ for $\nu\in\Lambda^+$ satisfying $\langle\nu,\alpha_i^\vee\rangle>0$
in \eqref{eq:Rel1}.
Details are left to the readers.
\end{proof}
It follows that for $a, b=1, 2$ we have
\begin{equation}
\label{eq:gen}
(\CS_1^{-1}\Dqab)(0)=
\langle
{\ell}_\varphi, {\deru}_x, 
{\deru}_{k_{a\lambda}}, {\sigma}_{b\mu}
\mid
\varphi\in(\CS_1^{-1}A_q)(0), x\in U_q(\Gn^+), 
\lambda, \mu\in\Lambda
\rangle.
\end{equation}
\subsection{}
Note that $O_q(B^+)$ is a $\Lambda$-graded algebra by the grading
\[
O_q(B^+)=\bigoplus_{\nu\in\Lambda}O_q(B^+)(\nu)
\]
given by
\[
O_q(B^+)(\nu)=
\{\xi\in O_q(B^+)\mid \xi\cdot h=\chi_\nu(h)\xi\;(h\in U_q(\Gh))\}.
\]
Since the  composite of the algebra homomorphisms
\[
A_q\hookrightarrow O_q(G)\xrightarrow{\res}O_q(B^+)
\]
sends $c_\lambda$ for $\lambda\in\Lambda^+$ to the invertible element $\pi^*(\chi_\lambda)\in O_q(B^+)$, 
we obtain an algebra homomorphism
\begin{equation}
\label{eq:theta}
\theta:\CS_1^{-1}A_q\to O_q(B^+).
\end{equation}
In fact $\theta$ is an  isomorphism of $\Lambda$-graded algebras (see \cite[Proposition 4.5]{TBB}).

Set
\[
U_q(\Gn^+)^\bigstar=
\bigoplus_{\gamma\in Q^+}
(U_q(\Gn^+)_{\gamma})^*\subset U_q(\Gn^+)^*.
\]
We have a  linear isomorphism
\begin{equation}
\label{eq:iota}
\iota:U_q(\Gn^+)^\bigstar\xrightarrow{\sim} O_q(B^+)(0)
\end{equation}
given by 
\[
\langle\iota(\xi),hx\rangle
=\varepsilon(h) \xi(x)
\qquad(\xi\in U_q(\Gn^+)^\bigstar,  h\in U_q(\Gh), x\in U_q(\Gn^+)).
\]
Note that $O_q(B^+)(0)$ is a 
subalgebra as well as a 
left $U_q(\Gb^+)$-submodule 
of $O_q(B^+)$.
We regard $U_q(\Gn^+)^\bigstar$ as an $\BF$-algebra and a left $U_q(\Gb^+)$-module via the identification \eqref{eq:iota}.
The identity element of $U_q(\Gn^+)^\bigstar$ is given by
$
\id_{U_q(\Gn^+)^\bigstar}=\varepsilon|_{U_q(\Gn^+)}
$.
\begin{remark}
\label{rem:D}
Using \eqref{eq:Dr} we see easily that $U_q(\Gn^+)^\bigstar$ is naturally isomorphic to the algebra $U_q(\Gn^-)^\op$ opposite to $U_q(\Gn^-)$.
\end{remark}

The multiplication of $O_q(B^+)$ gives the 
isomorphism
\begin{equation}
\label{eq:OB}
\jmath:
U_q(\Gn^+)^\bigstar\otimes O_q(H)
\xrightarrow{\sim} 
O_q(B^+)
\quad(\xi\otimes\chi\mapsto\iota(\xi)\pi^*(\chi))
\end{equation}
of $\BF$-modules.
For 
$\xi\in U_q(\Gn^+)^\bigstar$, $\chi\in O_q(H)$
we have
\[
\langle\jmath(\xi\otimes\chi),hx\rangle
=\langle\chi,h\rangle \xi(x)
\qquad(h\in U_q(\Gh), x\in U_q(\Gn^+)).
\]

We obtain an isomorphism
\begin{equation}
\label{eq:bt}
\vartheta=\jmath^{-1}\circ\theta:
\CS_1^{-1}A_q\to U_q(\Gn^+)^\bigstar\otimes O_q(H)
\end{equation}
of $\BF$-modules.
It satisfies
\begin{equation}
\vartheta((\CS_1^{-1}A_q)(\nu))=
U_q(\Gn^+)^\bigstar \otimes \chi_\nu
\qquad(\nu\in\Lambda).
\end{equation}
In particular, we have an isomorphism
\begin{equation}
\label{eq:bt0}
\vartheta_0:
(\CS_1^{-1}A_q)(0)\to U_q(\Gn^+)^\bigstar
\end{equation}
of $\BF$-algebras and $U_q(\Gb^+)$-modules 
satisfying $\vartheta(\varphi)=\vartheta_0(\varphi)\otimes1$ for $\varphi\in (\CS_1^{-1}A_q)(0)$.

We denote by 
$\BD_q$ the subalgebra of 
$\End_\BF(U_q(\Gn^+)^\bigstar\otimes O_q(H))$
corresponding to $(\CS_1^{-1}D_q)(0)\subset \End_\BF(\CS_1^{-1}A_q)$ under the isomorphism \eqref{eq:bt}.
In fact $\BD_q$ is a subalgebra of 
\[
\End_\BF(U_q(\Gn^+)^\bigstar)\otimes
\End_\BF(O_q(H))
\subset
\End_\BF(U_q(\Gn^+)^\bigstar\otimes O_q(H))
\]
by the following.
For $\psi\in \CS_1^{-1}A_q$ 
we have
\begin{subequations}
\label{subeq:ta}
\begin{align}
\vartheta(\sigma_\lambda(\psi))=&
(1\otimes \Bs_\lambda)(\vartheta(\psi))
&(\lambda\in\Lambda),
\\
\vartheta(\ell_\varphi(\psi))=&
(\Bm_{\vartheta_0(\varphi)}\otimes 1)(\vartheta(\psi))
&(\varphi\in(\CS_1^{-1}A_q)(0)),
\\
\vartheta(\deru_x(\psi))=&(\Bd_x\otimes1)(\vartheta(\psi))
&(x\in U_q(\Gn^+)),
\\
\vartheta(\deru_{k_\lambda}(\psi))=&
(\Bt_{-\lambda}\otimes \Bs_\lambda)(\vartheta(\psi))
&(\lambda\in\Lambda),
\end{align}
\end{subequations}
where 
\begin{subequations}
\begin{align}
\Bs_\lambda(\chi_\mu)=&q^{(\lambda,\mu)}\chi_\mu
\qquad&(\lambda, \mu\in\Lambda),
\\
\Bm_{\xi}(\xi')=&\xi\xi'
&(\xi. \xi'\in U_q(\Gn^+)^\bigstar),
\\
(\Bd_{x}(\xi))(x')=&\xi(x'x)
&(\xi\in U_q(\Gn^+)^\bigstar, x, x'\in U_q(\Gn^+)),
\\
\Bt_{\lambda}(\xi)=&q^{(\gamma,\lambda)}\xi&(\lambda\in\Lambda, \gamma\in Q^+, \xi\in (U_q(\Gn^+)_\gamma)^*).
\end{align}
\end{subequations}

We have
\begin{subequations}
\label{subeq:relA}
\begin{align}
&\Bm_{1}=\Bd_1=\Bt_0=\id,
\\
&\Bm_\xi\Bm_{\xi'}=\Bm_{\xi\xi'}
&(\xi, \xi'\in U_q(\Gn^+)^\bigstar),
\\
&\Bd_x\Bd_{x'}=\Bd_{xx''}
&(x, x'\in U_q(\Gn^+)),
\\
&\Bt_\lambda\Bt_\mu=\Bt_{\lambda+\mu}
&(\lambda, \mu\in\Lambda),
\\
\label{eq:relA5}
&\Bd_x\Bm_\xi=\sum_{(x)}\Bm_{x_{(0)}\cdot \xi}\Bd_{x_{(1)}}
&(x\in U_q(\Gn^+), \xi\in U_q(\Gn^+)^\bigstar),
\\
&\Bt_\lambda\Bd_x=q^{-(\gamma,\lambda)}\Bd_x\Bt_\lambda
&(x\in U_q(\Gn^+)_\gamma, \lambda\in\Lambda),
\\
&\Bt_\lambda\Bm_\xi=q^{(\gamma,\lambda)}\Bm_\xi\Bt_\lambda
&(\xi\in (U_q(\Gn^+)_\gamma)^*, \lambda\in\Lambda).
\end{align}
\end{subequations}
Here, \eqref{eq:relA5} makes sense by 
$
\Delta(U_q(\Gn^+))
\subset 
U_q(\Gb^+)\otimes U_q(\Gn^+)
$.
For $i\in I$ define $\xi_i\in (U_q(\Gn^+)_{\alpha_i})^*$ by 
$\xi_i(e_i)=1$.
By \eqref{eq:relA5} we obtain the relation 
\begin{equation}
\Bd_{e_i}\Bm_{\xi_j}
-q^{-(\alpha_i,\alpha_j)}
\Bm_{\xi_j}\Bd_{e_i}
=\delta_{ij}\qquad(i,j \in I).
\end{equation}

\subsection{}
For $a, b=1, 2$ we denote by 
$\BDqab$ the subalgebra of 
$\BD_q$
corresponding to $(\CS_1^{-1}\Dqab)(0)\subset (\CS_1^{-1}D_q)(0)$.
Set
\begin{align}
\BB_{q}=&\langle
\Bm_\xi, \Bd_x\mid
\xi\in U_q(\Gn^+)^\bigstar, 
x\in U_q(\Gn^+)
\rangle
\subset \End_\BF(U_q(\Gn^+)^\bigstar),
\\
\BE^c_q=&\langle\BB_q, \Bt_{\lambda}\mid\lambda\in c\Lambda\rangle
\subset \End_\BF(U_q(\Gn^+)^\bigstar)
\qquad&(c=1, 2),
\\
\BF[\Bs_{b\Lambda}]=&
\bigoplus_{\lambda\in b\Lambda}\BF\Bs_\lambda
\subset\End_\BF(O_q(H))
\qquad&(b=1, 2).
\end{align}
By \eqref{subeq:ta} we have
\begin{subequations}
\label{subeq:genB}
\begin{align}
\label{eq:genB1}
\BD_q=\BD_q^{11}
=&\BE_q^1\otimes\BF[\Bs_\Lambda]
,
\\
\label{eq:genB2}
\BD_q^{21}
=&\BE_q^2\otimes\BF[\Bs_\Lambda],
\\
\label{eq:genB3}
\BD_q^{12}
=&
\langle
\BB_q\otimes1, \Bt_\lambda\otimes\Bs_{-\lambda},
1\otimes\Bs_{2\mu}
\mid
\lambda, \mu\in\Lambda\rangle,
\\
\label{eq:genB4}
\BD_q^{22}
=&\BE_q^2\otimes\BF[\Bs_{2\Lambda}].
\end{align}
\end{subequations}

The algebra $\BB_q$ already appeared in \cite{Jo0}.
It also appeared in \cite{Kas} as a $q$-analogue of boson.
It is also called the quantum Weyl algebra in \cite{JoB}.

The following results are proved in \cite{Jo0}.
We include proofs here for the sake of readers.
\begin{proposition}
\label{prop:B}
\begin{itemize}
\item[(i)]
The multiplication of $\BB_q$ gives an isomorphism
\[
U_q(\Gn^+)^\bigstar\otimes U_q(\Gn^+)
\xrightarrow{\sim} \BB_q
\qquad(\xi\otimes x\mapsto\Bm_\xi\Bd_x)
\]
of $\BF$-modules.
\item[(ii)]
The algebra $\BB_q$ is an integral domain in the sense that 
$xy=0$ for $x,y\in \BB_q$ implies $x=0$ or $y=0$.
\item[(iii)]
The set $\BB_q^\times$ of  invertible elements in $\BB_q$ coincides with $\BF^\times$.
\end{itemize}
\end{proposition}
\begin{proof}
(i) The surjectivity follows from \eqref{subeq:relA}.
Hence it is sufficient to show the injectivity.
Take a basis $\{z_r\}$ of $U_q(\Gn^+)$ and $\gamma_r\in Q^+$ such that $z_r\in U_q(\Gn^+)_{\gamma_r}$.
Assume that there exists some non-zero element $\sum_r\xi_r\otimes z_r\in U_q(\Gn^+)^\bigstar\otimes U_q(\Gn^+)$ such that $\sum_r\Bm_{\xi_r}\Bd_{z_r}=0$.
Take $r_0$ such that $\xi_{r_0}\ne0$ and 
$\gamma_{r_0}-\gamma_r\notin Q^+$
for any $r\ne r_0$ satisfying $\xi_r\ne0$.
Define $\xi\in(U_q(\Gn^+)_{\gamma_{r_0}})^*$
by $\xi(z_r)=\delta_{rr_0}$.
Then we have $(\sum_r\Bm_{\xi_r}\Bd_{z_r})(\xi)=\xi_{r_0}\ne0$, which is a contradiction.
(i) is proved.

(ii), (iii) 
Recall that the monomials
$
e_{\beta_1}^{k_1}\dots e_{\beta_L}^{k_L}
$
for $(k_1,\dots, k_L)\in\BZ_{\geqq0}^L$ 
form a basis of $U_q(\Gn^+)$
(see \eqref{eq:root} for the notation).
As in \cite{DKP} 
we define a total order on $\BZ_{\geqq0}^L$ as follows.
For $\beta=\sum_{i\in I}m_i\alpha_i\in Q^+$ set
$\Ht(\beta)=\sum_i m_i$.
Then we have 
$(k_1,\dots, k_L)>(r_1,\dots, r_L)$ 
if and only if $\sum_tk_t\Ht(\beta_t)
>\sum_tr_t\Ht(\beta_t)$
or if 
$\sum_tk_t\Ht(\beta_t)
=\sum_tr_t\Ht(\beta_t)$
and there exists sum $t_0$ such that $k_{t}=r_t$ for $t=1, \dots, t_0-1$ and $k_{t_0}>r_{t_0}$.
For each $(k_1,\dots, k_L)\in\BZ_{\geqq0}^L$ 
set
\[
\BB_{q,\leqq(k_1,\dots, k_L)}
=
\sum_{(r_1,\dots, r_L)\leqq(k_1,\dots, k_L), \xi\in U_q(\Gn^+)^\bigstar}
\Bm_\xi\Bd_{e_{\beta_1}^{r_1}\dots e_{\beta_L}^{r_L}}\subset \BB_q.
\]
Then we have 
\[
\BB_{q,\leqq(k_1,\dots, k_L)}
\BB_{q,\leqq(k'_1,\dots, k'_L)}
\subset
\BB_{q,\leqq(k_1+k_1',\dots, k_L+k_L')}.
\]
Moreover, in the associated graded algebra 
$\gr\BB_q$ we have
\[
\gr(\Bd_{e_{\beta_s}})\gr(\Bd_{e_{\beta_t}})
=q^{(\beta_s,\beta_t)}
\gr(\Bd_{e_{\beta_t}})\gr(\Bd_{e_{\beta_s}})
\qquad(s<t),
\]
\[
\gr(\Bd_{e_{\beta_t}})\gr(\Bm_\xi)=
\gr(\Bm_{\Bt_{-\beta}(\xi)})
\gr(\Bd_{e_{\beta_t}})
\qquad(1\leqq t\leqq L,\xi\in U_q(\Gn^+)^\bigstar).
\]
Therefore, for non-zero elements $z, z'\in\BB_q$ with top terms 
$\Bm_\xi\Bd_{e_{\beta_1}^{k_1}\dots e_{\beta_L}^{k_L}}$, 
$\Bm_{\xi'}\Bd_{e_{\beta_1}^{r_1}\dots e_{\beta_L}^{r_L}}$ respectively, the top term of 
$zz'$ is given by 
$\Bm_{\xi(\Bt_{-k_1\beta_1-\dots k_L\beta_L}(\xi'))}\Bd_{e_{\beta_1}^{k_1+r_1}\dots e_{\beta_L}^{k_L+r_L}}$.
Hence our assertions (ii), (iii) for $\BB_q$ follow from the similar assertions for $U_q(\Gn^+)^\bigstar$.
By Remark \ref{rem:D}
they are equivalent to those for $U_q(\Gn^+)$, which are well-known
(It is also possible to show them by repeating our argument for $\BB_q$ above for the subalgebra of $\BB_q$ isomorphic to $U_q(\Gn^+)$).
\end{proof}

\section{Certain subalgebras}
\subsection{}
For $b=1, 2$ we set
\begin{equation}
\BF[\sigma_{b\Lambda}]
=\sum_{\lambda\in\Lambda}\BF\sigma_{b\lambda}
\subset D_q^{2b}
\subset D_q^{1b}.
\end{equation}
It is a subalgebra of $D_q^{2b}$ isomorphic to the group algebra of the abelian group $b\Lambda$.
\begin{theorem}
\label{thm:1}
\begin{itemize}
\item[(i)]
For $a=1, 2$ the canonical maps
\begin{equation}
\label{eq:thm1a}
\BF[\sigma_\Lambda]\otimes_{\BF[\sigma_{2\Lambda}]}
D_q^{a2}
\to D_q^{a1},
\qquad
D_q^{a2}
\otimes_{\BF[\sigma_{2\Lambda}]}
\BF[\sigma_\Lambda]
\to D_q^{a1}
\end{equation}
are bijective.
\item[(ii)]
For $b=1, 2$
the canonical maps
\begin{equation}
\label{eq:thm1b}
U_q(\Gg)\otimes_{U_q^e(\Gg)}
D_q^{2b}
\to D_q^{1b},
\qquad
D_q^{2b}
\otimes_{U_q^e(\Gg)}
U_q(\Gg)
\to D_q^{1b}
\end{equation}
with respect to the algebra homomorphisms
\[
U_q(\Gg)\to D_q^{1b}, \quad
U^e_q(\Gg)\to D_q^{2b}
\qquad(u\mapsto \deru_u)
\]
are bijective.
\end{itemize}
\end{theorem}

We note that the bijectivity of the first map in \eqref{eq:thm1a} 
is equivalent 
to that of the second map.
In fact the bijectivity of the first 
(resp.\ the second) map in \eqref{eq:thm1a} 
is equivalent to 
\[
D_q^{a1}=
\bigoplus_{\overline{\lambda}\in\Lambda/2\Lambda}
\sigma_\lambda D_q^{a2}
\qquad
\left(\text{resp.}\quad
D_q^{a1}=
\bigoplus_{\overline{\lambda}\in\Lambda/2\Lambda}
D_q^{a2} \sigma_\lambda
\right).
\]
However, we can easily show 
$\sigma_\lambda D_q^{a2}
=D_q^{a2} \sigma_\lambda$
by \eqref{eq:rel3}, \eqref{eq:rel5}, \eqref{eq:rel6}.

Similarly, the bijectivity of the first map in \eqref{eq:thm1b} is equivalent 
to that of the second map by 
\eqref{eq:UUe}.
\subsection{}
In order to verify Theorem \ref{thm:1}
we need some preliminary from the representation theory.

Assume $N(\Lambda,\Lambda)\subset 2\BZ$.
For $\lambda\in\frac12\Lambda$ 
we denote by 
$\BF_\lambda$ 
the one-dimensional 
$U_q(\Gb^+)$-module
corresponding to the character 
$
U_q(\Gb^+)\to\BF
$ 
defined by 
\[
z\mapsto \varepsilon(z) \quad(z\in U_q(\Gn^+)),
\qquad
k_\mu\mapsto q^{(\lambda,\mu)}
\quad(\mu\in\Lambda).
\]
We define a left $U_q(\Gg)$-module 
$M(\lambda)$ by
\[
M(\lambda)=U_q(\Gg)\otimes_{U_q(\Gb^+)}
\BF_\lambda.
\]
It is called the Verma module with  highest weight $\lambda$.
We have the weight space decomposition
\[
M(\lambda)=
\bigoplus_{\mu\in\lambda- Q^+}
M(\lambda)_\mu.
\]
We set 
\[
M^*(\lambda)=\bigoplus_{\mu\in-\lambda- Q^+}
(M(-\lambda)_\mu)^*
\subset
M(-\lambda)^*.
\]
It is a left $U_q(\Gg)$-module by 
\[
\langle um^*,m\rangle
=
\langle m^*,(Su)m\rangle
\qquad(m^*\in M^*(\lambda), u\in U_q(\Gg), 
m\in M(-\lambda)).
\]
We have also the weight space decomposition
\[
M^*(\lambda)=
\bigoplus_{\mu\in\lambda+ Q^+}
M^*(\lambda)_\mu.
\]

\begin{proposition}
\label{prop:key1}
For $\lambda\in\Lambda$ the subspace 
$U_q(\Gn^+)k_\lambda$ of $U_q(\Gg)$ is $\ad(U_q(\Gb^+))$-invariant.
Moreover, in the case $N(\Lambda,\Lambda)\subset 2\BZ$ we have an isomorphism
\[
U_q(\Gn^+)k_\lambda
\cong
M^*(\lambda/2)\otimes\BF_{-\lambda/2}
\]
of $U_q(\Gb^+)$-modules.
\end{proposition}
This follows from \cite[5.3.10, 7.1.1]{JoB}.
A more direct computational proof is given in \cite[Lemma 3.2]{T2}

For $a, b=1, 2$ we
define the adjoint action of $U_q(\Gg)$ on $\CS_1^{-1}D_q^{ab}$
by 
\[
\ad(u)(P)=\sum_{(u)}\deru_{u_{(0)}}P\deru_{Su_{(1)}}
\qquad(u\in U_q(\Gg), P\in \CS_1^{-1}D_q^{ab}).
\]
It respects the grading \eqref{eq:SDgr}.
For $u\in U_q(\Gg)$ have
\begin{subequations}
\label{eq:adD}
\begin{align}
\ad(u)(\ell_\varphi)=&\ell_{\deru_u(\varphi)}
&(\varphi\in \CS_1^{-1}A_q),
\\
\ad(u)(\deru_{u'})=&\deru_{\ad(u)(u')}
&(u'\in U_q(\Gg)),
\\
\ad(u)(\sigma_\lambda)=&\varepsilon(u)\sigma_\lambda
&(\lambda\in\Lambda).
\end{align}
\end{subequations}

We see easily that 
\[
D_q^{1b}=
\sum_{\overline{\lambda}\in\Lambda/2\Lambda}
D_q^{2b}\deru_{k_\lambda},
\qquad
\CS_1^{-1}D_q^{1b}=
\sum_{\overline{\lambda}\in\Lambda/2\Lambda}
(\CS_1^{-1}D_q^{2b})\deru_{k_\lambda},
\]
Moreover, we have
\[
\ad(U_q(\Gg))(D_q^{2b}\deru_{k_\lambda})\subset
D_q^{2b}\deru_{k_\lambda},
\qquad
\ad(U_q(\Gg))((\CS_1^{-1}D_q^{2b})\deru_{k_\lambda})\subset
(\CS_1^{-1}D_q^{2b})\deru_{k_\lambda}
\]
for any $\lambda\in\Lambda$
by \eqref{eq:adD} and \eqref{eq:adUe}.

\begin{proposition}
\label{prop:key2}
Assume $\lambda\in\Lambda\setminus 2\Lambda$.
For $b=1, 2$ 
$D_q^{2b}\deru_{k_\lambda}$ contains no non-zero $\ad(U_q(\Gb^+))$-invariant finite-dimensional subspace.
\end{proposition}
\begin{proof}
We may assume $N(\Lambda,\Lambda)\subset 2\BZ$ from the beginning.

Assume that 
$D_q^{2b}\deru_{k_\lambda}$ contains a non-zero $\ad(U_q(\Gb^+))$-invariant finite-dimensional subspace $V$.
Note that 
$D_q^{2b}\deru_{k_\lambda}$ is a direct sum of the 
$\ad(U_q(\Gb^+))$-invariant subspaces 
$D_q^{2b}(\nu)\deru_{k_\lambda}$
for $\nu\in\Lambda^+$.
Hence we may assume from the beginning that 
$V\subset D_q^{2b}(\nu)\deru_{k_\lambda}$
for some  $\nu\in\Lambda^+$.
Then $c_\nu^{-1}V$ is 
a non-zero $\ad(U_q(\Gb^+))$-invariant finite-dimensional subspace of $(\CS_1^{-1}D_q^{2b})(0){\deru}_{k_\lambda}$.
Set
\[
E=\langle
{\ell}_\varphi, {\deru}_x, 
 {\sigma}_{\mu}
\mid
\varphi\in(\CS_1^{-1}A_q)(0), x\in U_q(\Gn^+), 
\mu\in b\Lambda
\rangle
\subset (\CS_1^{-1}D_q^{2b})(0).
\]
By Proposition \ref{prop:Rel2} we have
\[
(\CS_1^{-1}D_q^{2b})(0)
=\bigcup_{\mu\in\Lambda^+}E{\deru}_{k_{-2\mu}},
\]
and hence 
by replacing $\lambda$ with $\lambda-2\mu$ for some $\mu\in\Lambda^+$
we may assume that
$c_\nu^{-1}V\subset E{\deru}_{k_\lambda}$
(note that $D_q^{2b}\deru_{k_\lambda}$ depends only on $\lambda+2\Lambda$).
By \eqref{subeq:genB} and Proposition \ref{prop:B} we have an isomorphism
\[
U_q(\Gn^+)^\bigstar\otimes U_q(\Gn^+)k_\lambda\otimes 
\BF[{\sigma}_{b\Lambda}]
\xrightarrow{\sim} E{\deru}_{k_\lambda}
\]
of $U_q(\Gb^+)$-modules.
Here, the action of $U_q(\Gb^+)$ on 
$U_q(\Gn^+)^\bigstar$
(resp.\
$U_q(\Gn^+)k_\lambda$, 
$\BF[{\sigma}_{b\Lambda}]$) is 
given by the identification \eqref{eq:iota}
(resp.\ the adjoint action, resp.\ trivial action).
Hence $U_q(\Gn^+)^\bigstar\otimes U_q(\Gn^+)k_\lambda$ 
contains a non-zero $U_q(\Gb^+)$-invariant finite-dimensional subspace.
Since $U_q(\Gn^+)^\bigstar$ admits a filtration 
\[
0=F_{-1}\subset F_0\subset F_1\subset\cdots\subset
U_q(\Gn^+)^\bigstar
\]
by $U_q(\Gb^+)$-invariant subspaces such that
\[
U_q(\Gn^+)^\bigstar=\bigcup_r F_r, \qquad
F_r/F_{r-1}\cong\BF_{\mu_r}
\]
for $\mu_r\in Q$, 
there exists $r$ such that 
$(F_r\otimes U_q(\Gn^+)k_\lambda)/(
F_{r-1}\otimes U_q(\Gn^+)k_\lambda)$
contains a non-zero $U_q(\Gb^+)$-invariant finite-dimensional subspace.
Therefore, $U_q(\Gn^+)k_\lambda$
contains a non-zero $U_q(\Gb^+)$-invariant finite-dimensional subspace by 
\[
(F_r\otimes U_q(\Gn^+)k_\lambda)/(
F_{r-1}\otimes U_q(\Gn^+)k_\lambda)
\cong \BF_{\mu_r}\otimes
U_q(\Gn^+)k_\lambda.
\]
Hence by Proposition \ref{prop:key1}  $M^*(\lambda/2)$ contains a non-zero $U_q(\Gb^+)$-invariant finite-dimensional subspace.
We see easily that $\dim U_q(\Gb^-)m<\infty$ for any $m\in M^*(\lambda/2)$, and hence 
$M^*(\lambda/2)$ contains a non-zero $U_q(\Gg)$-invariant finite-dimensional subspace.
This contradicts with the assumption $\lambda/2
\notin \Lambda$.
In fact, the central character of $M^*(\lambda/2)$ 
does not coincide with that of any finite-dimensional irreducible $U_q(\Gg)$-module.
\end{proof}
The proof of the following result was inspired by that of \cite[Corollary 3.3]{Jo1}.
\begin{proposition}
\label{prop:key3}
We have
\[
D_q^{1b}=
\bigoplus_{\overline{\lambda}\in\Lambda/2\Lambda}
D_q^{2b}\deru_{k_\lambda}.
\]
\end{proposition}
\begin{proof}
We show  that 
for any $\Gamma\subset \Lambda/2\Lambda$  the sum 
$
\sum_{\overline{\lambda}\in\Gamma}
D_q^{2b}\deru_{k_\lambda}
$
is a direct sum by induction on $|\Gamma|$.
The case $|\Gamma|=1$ being obvious we assume that $|\Gamma|\geqq 2$.
Since $\deru_{k_\mu}$ is invertible for $\mu\in\Lambda$, 
the sum $
\sum_{\overline{\lambda}\in\Gamma}
D_q^{2b}\deru_{k_\lambda}
$ is a direct sum if and only if 
the sum
$
\sum_{\overline{\lambda}\in\Gamma}
D_q^{2b}\deru_{k_{\lambda+\mu}}
$
is a direct sum.
Hence we may assume from the beginning that $\overline{0}\in\Gamma$.
Set $\Gamma'=\Gamma\setminus\{\overline{0}\}$.
By the hypothesis of induction 
we have
\[
\sum_{\overline{\lambda}\in\Gamma'}
D_q^{2b}\deru_{k_\lambda}
=
\bigoplus_{\overline{\lambda}\in\Gamma'}
D_q^{2b}\deru_{k_\lambda}.
\]
Hence it is sufficient to show 
\[
D_q^{2b}\cap
\left(\bigoplus_{\overline{\lambda}\in\Gamma'}
D_q^{2b}\deru_{k_\lambda}\right)=\{0\}.
\]
Assume there exists a non-zero  element $P\in D_q^{2b}\cap
\left(\bigoplus_{\overline{\lambda}\in\Gamma'}
D_q^{2b}\deru_{k_\lambda}\right)$.
By \eqref{eq:Ufe} there exists some $\mu\in\Lambda^+$ such that $\dim\ad(U_q(\Gg))(P\deru_{k_{-2\mu}})<\infty$.
This implies that 
$\bigoplus_{\overline{\lambda}\in\Gamma'}
D_q^{2b}\deru_{k_\lambda}$ contains a non-zero $\ad(U_q(\Gb^+))$-invariant finite-dimensional subspace. This contradicts with Proposition \ref{prop:key2}.
\end{proof}

\subsection{}
Let us show Theorem \ref{thm:1}.
The assertion (ii) follows from Proposition \ref{prop:key3}.
Let us show (i).
We first consider the case $a=2$.
By the injectivity of $D_q^{2b}\to\CS_1^{-1}D_q^{2b}$ for $b=1, 2$ 
it is sufficient to show 
\[
\CS_1^{-1}D_q^{21}=\bigoplus_{\overline{\lambda}\in\Lambda/2\Lambda}
(\CS_1^{-1}D_q^{22})\sigma_{\lambda}.
\]
Hence it is suffcient to show 
\[
(\CS_1^{-1}D_q^{21})(0)=\bigoplus_{\overline{\lambda}\in\Lambda/2\Lambda}
((\CS_1^{-1}D_q^{22})(0))\sigma_{\lambda}.
\]
This is equivalent to 
\[
\BD_q^{21}=\bigoplus_{\overline{\lambda}\in\Lambda/2\Lambda}
\BD_q^{22}(1\otimes\Bs_\lambda),
\]
which easily follows from \eqref{eq:genB2}, \eqref{eq:genB4}.
Finally,  we obtain (i) for $a=1$ by
\[
D_q^{11}
=\bigoplus_{\overline{\lambda}\in\Lambda/2\Lambda}
D_q^{21}\deru_{k_\lambda}
=\bigoplus_{\overline{\lambda},\overline{\mu}\in\Lambda/2\Lambda}
D_q^{22}\sigma_\mu\deru_{k_\lambda}
=\bigoplus_{\overline{\lambda},\overline{\mu}\in\Lambda/2\Lambda}
D_q^{22}\deru_{k_\lambda}\sigma_\mu
=\bigoplus_{\overline{\mu}\in\Lambda/2\Lambda}
D_q^{12}\sigma_\mu.
\]
The proof of 
Theorem \ref{thm:1} is now complete.
\section{Local expressions}
In this section we consider $\BD_q$ and its subalgebras.

We set 
\begin{align}
\Bt_{2\Lambda^+}=&
\{\Bt_{\lambda}\mid\lambda\in2\Lambda^+\},
\\
\BF[\Bt_{2\Lambda^+}]
=&\bigoplus_{\lambda\in 2\Lambda^+}\BF\Bt_\lambda,
\\
\BF[\Bt_{c\Lambda}]
=&\bigoplus_{\lambda\in c\Lambda}\BF\Bt_\lambda
\qquad&(c=1, 2).
\end{align}
\begin{theorem}
\label{thm:2}
\begin{itemize}
\item[(i)]
We have $\Bt_{2\Lambda^+}\subset\BB_q$.
\item[(ii)]
The multiplicative subset 
$\Bt_{2\Lambda^+}$
of $\BB_q$ satisfies the left and right Ore conditions.
\item[(iii)]
We have 
\begin{equation*}
\BE_q^2\cong
(\Bt_{2\Lambda^+})^{-1}\BB_q
\cong 
\BF[\Bt_{2\Lambda}]\otimes_{\BF[\Bt_{2\Lambda^+}]}\BB_q
\cong 
\BB_q
\otimes_{\BF[\Bt_{2\Lambda^+}]}
\BF[\Bt_{2\Lambda}].
\end{equation*}
\item[(iv)]
We have 
\begin{equation*}
\BE_q^1\cong
\BF[\Bt_{\Lambda}]\otimes_{\BF[\Bt_{2\Lambda}]}\BE_q^2
\cong 
\BE_q^2
\otimes_{\BF[\Bt_{2\Lambda}]}
\BF[\Bt_{\Lambda}].
\end{equation*}

\end{itemize}
\end{theorem}
\begin{proof}
(i) is a consequence of Proposition \ref{prop:Rel2}.
(ii) follows from \eqref{subeq:relA} and Proposition \ref{prop:B} (ii).
Since elements of $\Bt_{2\Lambda^+}$ are invertible in $\BE^2_q$, we have a canonical surjective algebra homomorphism 
$(\Bt_{2\Lambda^+})^{-1}\BB_q
\to\BE_q^2$.
Its injectivity follows from the exactitude of the localization functor.
(iii) is proved.
To prove (iv) it is sufficient to show 
\[
\BE_q^1=
\bigoplus_{\overline{\lambda}\in\Lambda/2\Lambda}
\Bt_\lambda\BE_q^2
=
\bigoplus_{\overline{\lambda}\in\Lambda/2\Lambda}
\BE_q^2\Bt_\lambda,
\]
which is equivalent to
\[
\BD_q^{11}=
\bigoplus_{\overline{\lambda}\in\Lambda/2\Lambda}
(\Bt_\lambda\otimes1)\BD_q^{21}
=
\bigoplus_{\overline{\lambda}\in\Lambda/2\Lambda}
\BD_q^{21}(\Bt_\lambda\otimes1).
\]
This  is again equivalent to 
\[
\BD_q^{11}=
\bigoplus_{\overline{\lambda}\in\Lambda/2\Lambda}
(\Bt_\lambda\otimes\Bs_{-\lambda})\BD_q^{21}
=
\bigoplus_{\overline{\lambda}\in\Lambda/2\Lambda}
\BD_q^{21}(\Bt_\lambda\otimes\Bs_{-\lambda}),
\]
which follows easily from Theorem \ref{thm:1} (ii) for $b=1$.
\end{proof}
\begin{remark}
By Proposition \ref{prop:B} $\Bt_{\lambda}$ 
for $\lambda\in2\Lambda^+$ is invertible in $\BB_q$ only if $\lambda=0$.
\end{remark}
\begin{corollary}
We have
\begin{equation*}
\BD_q\cong
(\BF[\Bt_{\Lambda}]\otimes_{\BF[\Bt_{2\Lambda^+}]}\BB_q)
\otimes\BF[\Bs_\Lambda]
\cong 
(\BB_q
\otimes_{\BF[\Bt_{2\Lambda^+}]}
\BF[\Bt_{\Lambda}])
\otimes\BF[\Bs_\Lambda].
\end{equation*}
\end{corollary}
\section{Category of $D$-modules}
\subsection{}
Let us recall basic facts on the quantized flag manifold (see \cite{LR}, \cite{TBB}).
The quantized flag manifold $\CB_q$ is 
a non-commutative projective scheme defined by
 \[
 \CB_q=\Proj(A_q).
 \]
Namely the category $\Mod(\CO_{\CB_q})$ of quasi-coherent 
$\CO_{\CB_q}$-modules is given by
\[
\Mod(\CO_{\CB_q})
=
\Mod_\Lambda(A_q)/\Tor_{\Lambda^+}(A_q),
\]
where $\Mod_\Lambda(A_q)$ is the category of $\Lambda$-graded $A_q$-modules, and $\Tor_{\Lambda^+}(A_q)$
is its full subcategory consisting of $M\in \Mod_\Lambda(A_q)$ such that 
for any $m\in M$ there exists some $\nu\in\Lambda^+$ satisfying $A(\nu+\nu')m=0$ for any $\nu'\in\Lambda^+$.
We denote by
\begin{equation}
\varpi:
\Mod_\Lambda(A_q)\to\Mod(\CO_{\CB_q})
\end{equation}
the canonical functor.

For $w\in W$ 
set $R_{w,q}=(\CS_w^{-1}A_q)(0)$, and  define a non-commutative affine scheme $U_{w,q}$ by $U_{w,q}=\Spec(R_{w,q})$.
Namely the category $\Mod(\CO_{U_{w,q}})$ of quasi-coherent $\CO_{U_{w,q}}$-modules is given by the category $\Mod(R_{w,q})$ of $R_{w,q}$-modules.
The natural functor
\[
\Mod_\Lambda(A_q)\to\Mod(R_{w,q})
\qquad(M\mapsto (\CS_w^{-1}M)(0))
\]
induces an exact functor
\[
\Mod(\CO_{\CB_q})\to\Mod(\CO_{U_{w,q}})
\qquad(\CM\mapsto \CM|_{U_{w,q}}),
\]
by which 
we have an  affine open covering 
\[
\CB_q=\bigcup_{w\in W}U_{w,q}.
\]
In fact 
 we have the following patching property by \cite{LR}:
\begin{equation}
\label{eq:patch}
\CM\in \Mod(\CO_{\CB_q}), \;\;
\CM|_{U_{w,q}}=0\;\;(\forall w\in W)
\;\;
\Longrightarrow
\CM=0.
\end{equation}

Now let us consider $D$-modules on $\CB_q$.
Denote by $\Mod_\Lambda(D_q)$ the category of $\Lambda$-graded $D_q$-modules.
For $\theta\in\Hom_{\alg}(\BF[\sigma_\Lambda],\BF)$ we define  $\Mod_\Lambda(D_q,\theta)$
to be the full subcategory of 
$\Mod_\Lambda(D_q)$ 
consisting of $M\in \Mod_\Lambda(D_q)$ satisfying 
$
\sigma_\mu|_{M(\nu)}=\theta(\sigma_\mu)q^{(\mu,\nu)}\id$ for any $\mu, \nu\in\Lambda$.
Set 
\begin{align*}
\Mod(\DD_{\CB_q})
=&\Mod_\Lambda(D_q)/(\Mod_\Lambda(D_q)\cap\Tor_{\Lambda^+}(A_q)),
\\
\Mod(\DD_{\CB_q,\theta})
=&\Mod_\Lambda(D_q,\theta)/(\Mod_\Lambda(D_q,\theta)\cap\Tor_{\Lambda^+}(A_q))
\quad(\theta\in\Hom_{\alg}(\BF[\sigma_\Lambda],\BF)).
\end{align*}
For $a, b=1, 2$ we can similarly define 
for $\theta\in \Hom_{\alg}(\BF[\sigma_{b\Lambda}],\BF)$ the categories
$\Mod(\DD^{ab}_{\CB_q})$, $\Mod(\DD^{ab}_{\CB_q,\theta})$ using $D_q^{ab}$ instead of $D_q$.
Let $\res: \Hom_{\alg}(\BF[\sigma_{\Lambda}],\BF)\to \Hom_{\alg}(\BF[\sigma_{2\Lambda}],\BF)$ be the natural homomorphism given by the restriction.
By Theorem \ref{thm:1} (i) we see easily the following.
\begin{proposition}
\label{prop:eo}
For $a=1, 2$ and $\theta\in\Hom_{\alg}(\BF[\sigma_{\Lambda}],\BF)$ we have a natural equivalence
\[
\Mod(\DD^{a1}_{\CB_q,\theta})
\cong
\Mod(\DD^{a2}_{\CB_q,\res(\theta)})
\]
of categories.
\end{proposition}
\subsection{}
Let 
\[
\BF[\Lambda]=\bigoplus_{\lambda\in\Lambda}\BF e(\lambda)
\]
be the group algebra of $\Lambda$, and set
\begin{equation}
\zDq=A_q\otimes U_q(\Gg)\otimes\BF[\Lambda].
\end{equation}
Then it is easily seen that $\zDq$ is endowed 
with a $\Lambda$-graded $\BF$-algebra structure by 
the multiplication 
\begin{align*}
&
(\varphi\otimes u\otimes e(\lambda))
(\psi\otimes v\otimes e(\mu))
=
\sum_{(u)}
q^{(\lambda,\nu)}\varphi(u_{(0)}\cdot\psi)\otimes
u_{(1)}v\otimes e(\lambda+\mu)
\\
&\qquad(\varphi\in A_q, \nu\in\Lambda^+, \psi\in A_q(\nu), u, v\in U_q(\Gg), \lambda, \mu\in \Lambda)
\end{align*}
and the grading $\zDq(\lambda)=A_q(\lambda)\otimes U_q(\Gg)\otimes\BF[\Lambda]$.
We regard $A_q$, $U_q(\Gg)$, $\BF[\Lambda]$ as subalgebras of $\zDq$ by identifying them with
$A_q\otimes1\otimes1$, $1\otimes U_q(\Gg)\otimes1$, $1\otimes 1\otimes \BF[\Lambda]$ respectively.
Then the multiplication of $\zDq$ gives isomorphisms 
\[
A_q\otimes U_q(\Gg)\otimes\BF[\Lambda]\xrightarrow{\sim} \zDq,
\qquad
\BF[\Lambda]\otimes U_q(\Gg)\otimes A_q
\xrightarrow{\sim} \zDq
\]
of $\BF$-modules.
In $\zDq$ we have
\begin{subequations}
\label{subeq:relX}
\begin{align}
\label{eq:relX4}
u\varphi=&\sum_{(u)}
(u_{(0)}\cdot\varphi)u_{(1)}
&(u\in U_q(\Gg),\; \varphi\in {A_q}),
\\
\label{eq:X4a}
\varphi u=&
\sum_{(u)}
u_{(1)}((S^{-1}u_{(0)})\cdot\varphi)
&
(u\in U_q(\Gg),\; \varphi\in {A_q}),
\\
\label{eq:relX5}
e(\lambda)u=&
ue(\lambda)
&(\lambda\in\Lambda,\;  u\in U_q(\Gg)),
\\
\label{eq:relX6}
e(\lambda)\varphi=&
q^{(\lambda,\nu)}\varphi e(\lambda)
&(\lambda\in\Lambda,\; \nu\in\Lambda^+,\;\varphi\in {A_q}(\nu)).
\end{align}
\end{subequations}

We have a surjective $\BF$-algebra homomorphism
\[
{}^0p:\zDq\to D_q
\]
given by 
\[
\varphi
\mapsto\ell_\varphi\;\;(\varphi\in A_q), \qquad
u\mapsto\deru_u\;\;(u\in U_q(\Gg)), \qquad
e(\lambda)\mapsto\sigma_\lambda\;\;(\lambda\in\Lambda).
\]
For $\varphi\in A_q(\nu)_\mu$ with $\nu\in\Lambda^+$, $\mu\in\Lambda$ set 
\begin{equation}
\Omega(\varphi)=
\sum_r(y_r\cdot\varphi){x_r}{k_{-2\mu}}e(2\nu)
-\sum_r((Sx_r)\cdot\varphi)(y_rk_{\beta_r})
\in \zDq
\end{equation}
(see \eqref{eq:canon} for the notation).
By Proposition \ref{prop:Rel1} we have $\Omega(\varphi)\in\Ker({}^0p)$.
We set
\begin{equation}
\tilde{D}_q=
\zDq/I_q,
\qquad
I_q=\sum_{\nu\in\Lambda^+,\mu\in\Lambda}\sum_{\varphi\in A_q(\nu)_\mu}
\zDq\Omega(\varphi)\zDq.
\end{equation}
Then ${}^0p$ induces a surjective homomorphism
$\tilde{p}:\tilde{D}_q\to D_q$
of $\Lambda$-graded algebras.
\begin{lemma}
For any $w\in W$ 
the multiplicative set $\CS_w$ satisfies the left and right Ore conditions both in $\zDq$ and $\tilde{D}_q$.
\end{lemma}
\begin{proof}
We only show the left Ore conditions:
\begin{subequations}
\begin{align}
\label{eq:Ore1}
(a,s)\in R\times \CS_w\Longrightarrow\exists(b,t)\in R\times\CS_w\;\;\text{s.t.}\;\; ta=bs,
\\
\label{eq:Ore2}
(a,s)\in R\times \CS_w, \; as=0
\Longrightarrow\exists \;t\in\CS_w\;\;\text{s.t.}\;\; ta=0
\end{align}
\end{subequations}
for $R=\zDq, \tilde{D}_q$.
The right Ore conditions are proved similarly.

We see easily from \eqref{subeq:relX} that 
\eqref{eq:Ore1} holds for $R=\zDq$.
Moreover, since $\zDq$ is a free right $A_q$-module 
and since $A_q$ is an integral  domain, we obtain 
\begin{align}
\label{eq:Ore3}
(a,s)\in \zDq\times \CS_w, \; as=0
\Longrightarrow a=0,
\end{align}
which implies \eqref{eq:Ore2} for $R=\zDq$.

The condition \eqref{eq:Ore1} for $R=\tilde{D}_q$ follows from that for $R=\zDq$.
Let us finally show \eqref{eq:Ore2} for $R=\tilde{D}_q$.
We first note
\begin{align}
\label{eq:Ore4}
(b,s)\in I_q\times \CS_w\Longrightarrow\exists(c,t)\in I_q\times\CS_w\;\;\text{s.t.}\;\; tb=cs,
\end{align}
which is a consequence of \eqref{subeq:relX} and \cite[(4.11)]{T1}.
Assume $\overline{as}=0$ in $\tilde{D}_q$ for $(a,s)\in\zDq\times\CS_w$.
Then we have $as\in I_q$.
By \eqref{eq:Ore4} there exists $(c,t)\in I_q\times\CS_w$ such that $tas=cs$.
Then by \eqref{eq:Ore3} we have $ta=c\in I_q$.
Hence $\overline{ta}=0$ in $\tilde{D}_q$.
\end{proof}
\begin{lemma}
\label{lem:KT}
We have
$\Ker(\tilde{p})\in \Tor_{\Lambda^+}(A_q)$.
\end{lemma}
\begin{proof}
We obtain from the exact sequence
\[
0\to \Ker(\tilde{p})\to \tilde{D}_q\to D_q\to 0
\]
in $\Mod_\Lambda(A_q)$
the exact sequence 
\[
0\to \varpi(\Ker(\tilde{p}))\to
\varpi(\tilde{D}_q)\to \varpi(D_q)\to 0
\]
in $\Mod(\CO_{\CB_q})$.
Hence it is sufficient to show that 
the morphism 
$\varpi(\tilde{D}_q)\to \varpi(D_q)$
in $\Mod(\CO_{\CB_q})$ is an isomorphism.
By \eqref{eq:patch} we have only to verify that 
$\varpi(\tilde{D}_q)|_{U_{w,q}}\to \varpi(D_q)|_{U_{w,q}}$ is an isomorphism for any $w\in W$.
By
$\varpi(M)|_{U_{w,q}}\cong(\CS_w^{-1}M)(0)$ for $M\in \Mod_\Lambda(A_q)$
 it is sufficient to show 
that $(\CS_w^{-1}\tilde{D}_q)(0)
\to(\CS_w^{-1}{D}_q)(0)$ is an isomorphism.
Using the action of the braid group on $D_q$ and $\tilde{D}_q$ we are reduced to the case $w=1$ (see \cite[4.2]{T1}).
Namely, it is sufficient to show that 
the surjective algebra homomorphism
 $(\CS_1^{-1}\tilde{D}_q)(0)
\to(\CS_1^{-1}{D}_q)(0)$ is injective.

For $\lambda\in\Lambda$ set
\[
\tilde{\Bs}_\lambda=\overline{e(\lambda)}
\in (\CS_1^{-1}\tilde{D}_q)(0),
\qquad
\tilde{\Bt}_\lambda=\overline{k_{-\lambda}e(\lambda)}\in (\CS_1^{-1}\tilde{D}_q)(0),
\]
where the canonical homomorphism $(\CS_1^{-1}(\zDq))(0)\to (\CS_1^{-1}\tilde{D}_q)(0)$ is denoted by $a\mapsto \overline{a}$.
We also set
\begin{align*}
\tilde{\BB}_q=&\langle
\overline{(\CS_1^{-1}A_q)(0)}, 
\overline{U_q(\Gn^+)}\rangle
\subset (\CS_1^{-1}\tilde{D}_q)(0),
\\
\tilde{\BE}^1_q=&
\langle \tilde{\BB}_q, \tilde{\Bt}_\lambda\mid
\lambda\in\Lambda\rangle
\subset (\CS_1^{-1}\tilde{D}_q)(0).
\end{align*}
Then by the proofs of Proposition \ref{prop:Rel2}, Proposition \ref{prop:B} we see easily that 
\begin{align*}
(\CS_1^{-1}\tilde{D}_q)(0)=\sum_{\lambda\in\Lambda}
\tilde{\BE}^1_q\tilde{\Bs}_\lambda,
\qquad
\tilde{\BE}^1_q=\sum_{\lambda\in\Lambda}\tilde{\Bt}_\lambda\tilde{\BB}_q,
\qquad
\tilde{\Bt}_{\lambda}\in\tilde{\BB}_q\quad(\lambda\in 2\Lambda^+).
\end{align*}
Moreover, the canonical homomorphism
\begin{align*}
(\CS_1^{-1}A_q)(0)\otimes U_q(\Gn^+)\to \tilde{\BB}_q
\end{align*}
is surjective.
Therefore, we conclude that the canonical surjective homomorphism $(\CS_1^{-1}\tilde{D}_q)(0)
\to(\CS_1^{-1}{D}_q)(0)$ is injective by
\eqref{eq:genB1}, Proposition \eqref{prop:B} (i), Theorem \ref{thm:2}.
\end{proof}

Define a category $\Mod(\tilde{\DD}_{\CB_q})$
similarly to $\Mod({\DD}_{\CB_q})$
using $\tilde{D}_q$ instead of $D_q$.
\begin{theorem}
\label{thm:3}
We have an equivalence 
\[
\Mod(\tilde{\DD}_{\CB_q})\cong
\Mod({\DD}_{\CB_q})
\]
of abelian categories.
\end{theorem}
\begin{proof}
The surjective homomorphism $\tilde{p}:\tilde{D}_q\to D_q$ gives a functor $\tilde{p}^*:
\Mod_\Lambda(D_q)
\to
\Mod_\Lambda(\tilde{D}_q)
$.
We have also a functor in the opposite direction 
\[
\Mod_\Lambda(\tilde{D}_q)
\to\Mod_\Lambda(D_q)
\]
given by $M\mapsto M/\Tor(M)$, where
\[
\Tor(M)=\{m\in M\mid
\exists \nu\in\Lambda^+\;\text{s.t.}\;
A_q(\nu+\nu')m=0\;(\forall\nu'\in\Lambda^+)\}.
\]
It is easily seen by Lemma \ref{lem:KT} that those functors induce the equivalence 
$
\Mod(\tilde{\DD}_{\CB_q})\cong
\Mod({\DD}_{\CB_q})$.
\end{proof}
\begin{remark}
The analogues of  Theorem \ref{thm:1}, 
Theorem \ref{thm:2}, Proposition \ref{prop:eo}, 
Theorem \ref{thm:3}
also hold 
in the situation when $q$ is specialized to $\zeta\in\BC^\times$ satisfying $\zeta^{(\alpha_i,\alpha_i)}\ne1$ for any $i\in I$ 
(see \cite{T1} for a more precise formulation).
This is easily seen by checking that 
the corresponding statements over 
$\BA=\BQ[q^{\pm1/N},(q_i-q_i^{-1})^{-1}\mid i\in I]$ 
follow from those over $\BF$.
Details are omitted.
\end{remark}

\bibliographystyle{unsrt}

\end{document}